\documentclass[11pt]{article}
\usepackage{graphicx}
\usepackage{amssymb}
\usepackage{amsmath,amsthm}
\usepackage{verbatim}
\usepackage{amsfonts}
\usepackage[T1]{fontenc}
\usepackage{color}

\newtheorem{thm}{Theorem}

\newtheorem{cor}[thm]{Corollary}
\newtheorem{defn}[thm]{Definition}
\newtheorem{lem}[thm]{Lemma}

\newcommand{\C}{\mathcal{C}}

\newcommand{\G}{\mathcal{G}}

\newcommand{\List}{\mathcal{L}}

\newcommand{\N}{\mathcal{N}}

\newcommand{\Pa}{\mathcal{P}}
\newcommand{\La}{\mathcal{L}}

\newcommand{\bP}{\mathbb{P}}
\newcommand{\bE}{\mathbb{E}}

\newcommand{\bF}{\mathbb{F}}
\newcommand{\bN}{\mathbb{N}}

\newcommand{\PG}{$PG(2,\bF_q)$}
\newcommand{\dcn}{Distinguishing Chromatic Number \hspace{.1cm}}

\title{$\chi_D(G)$, $|Aut(G)|$, and a variant of the Motion Lemma }
\author{Niranjan Balachandran and Sajith Padinhatteeri,\\ Department of Mathematics,\\
Indian Institute of Technology Bombay,\\ Mumbai, India. }
\date{}
\begin{document}
\maketitle

\begin{abstract}
The \textit{Distinguishing Chromatic Number}  of a graph $G$, denoted $\chi_D(G)$, was first defined in \cite{collins} as the minimum number of colors needed to properly color $G$ such that no non-trivial automorphism $\phi$ of the graph $G$ fixes each color class of $G$. In this paper, 
\begin{enumerate}
\item We prove a lemma that may be considered a variant of the Motion lemma of \cite{RS} and use this to give examples of several families of graphs which satisfy $\chi_D(G)=\chi(G)+1$.
\item We give an example of families of graphs that admit large automorphism groups in which every proper coloring is distinguishing. We also describe families of graphs with (relatively) very small automorphism groups which satisfy $\chi_D(G)=\chi(G)+1$, for arbitrarily large values of $\chi(G)$.
\item We describe non-trivial families of bipartite graphs that satisfy $\chi_D(G)>r$ for any positive integer $r$.\end{enumerate} 
\end{abstract}

\textbf{Keywords:}
Distinguishing Chromatic Number, Automorphism group of a graph, Motion lemma,  Weak product of graphs.\\

2010 AMS Classification Code: 05C15, 05C25, 05C76, 05C80. 

\section{Introduction}

For a graph $G=(V,E)$ let us denote by $Aut(G)$, its full automorphism group. A labeling of vertices of a graph $G, h : V(G) \rightarrow  \{1, \hdots, r \}$
is said to be {\bf distinguishing} (or $r$-distinguishing)  provided no nontrivial automorphism of the graph preserves all of the vertex labels. The \textbf{Distinguishing number} of the graph $G$, denoted by $D(G)$, is the minimum $r$ such that $G$ has an $r$-distinguishing labeling (see \cite{AK}).

Collins and Trenk introduced the notion of the \textbf{Distinguishing Chromatic Number} in \cite{collins}, as the minimum number of colors $r$, needed to color the vertices of the graph so that the coloring is both proper and distinguishing. In other words, the \dcn is the least integer $r$ such that the vertex set can be partitioned into sets $V_1,V_2,\ldots, V_r$ such that each $V_i$ is independent in $G$, and for every $I\neq\pi\in Aut(G)$ there exists some color class $V_i$ such that $\pi(V_i)\neq V_i$.

The problem of determining the distinguishing chromatic number of a graph $G$, or at least good bounds for it,  has been one of considerable interest in recent times (\cite{CHT,LS,collins,kneser,Hemanshu}). Clearly, the notion of the distinguishing chromatic number begins to get more interesting only if the graph admits a large group of automorphisms, in which case, it can vary substantially from the usual chromatic number. It is known (see \cite{collins}) that $\chi_D(G)=|V|$ if and only if $G$ is complete multipartite. Consequently, it is  simple  to see that there exist graphs $G$ with $\chi(G) = k, \chi_D(G) = l+k$, for any $k,l,$ since for instance, a disjoint union of a clique of size $k$ and $K_{1,l}$ achieves the same. Some upper bounds for  $\chi_D(G)$ (for instance, a version of Brooks'  theorem for the distinguishing chromatic number) appear in \cite{CHT}, which also includes the inequality $\chi_D(G)\le D(G)\chi(G)$. However, in many interesting large naturally occurring families of graphs, we have  $\chi_D(G) \leq \chi(G)+1$ (see \cite{Hemanshu,kneser,collins,CHT}).      

In this paper, we seek to address three aspects of the problem of determining $\chi_D(G)$ for a given graph $G$. Firstly, we prove a lemma that may be considered a variant of what is now well known as the motion lemma, first introduced in \cite{RS}. The motion lemma basically says that  if every nontrivial automorphism of a graph moves `many' vertices then its distinguishing number is small. A similar lemma also appears in the context of graph endomorphisms and `endomorphism breaking' in \cite{eb}. As a result of our variant of the Motion lemma, we give examples of several families of graphs $G$ satisfying $\chi_D(G)=\chi(G)+1$. 

Secondly, we  contrast the relation between the size of the automorphism group $Aut(G)$ of a graph with its distinguishing chromatic number $\chi_D(G)$. A result describing an upper bound for $\chi_D(G)$ in terms of the prime factors of $|Aut(G)|$ appears in \cite{CHT}. However, our perspective is somewhat different. We demonstrate instances of families of graphs $G$ such that $G$ have large chromatic number, and $\chi_D(G)=\chi(G)+1$ even though $|Aut(G)|$ is not very large (we have $|Aut(G)|=O(|V|^{3/2})$). As a contrast, we also demonstrate a family of graphs with arbitrarily large chromatic number, with `super large' automorphism groups for which \textit{every} proper coloring of $G$ with $\chi(G)$ colors is in fact distinguishing. This latter example also addresses a point raised in \cite{Hemanshu} and these contrasting results indicate that the relation between $|Aut(G)|$ and $\chi_D(G)$ can tend to be  somewhat haphazard.

Finally, as we indicated earlier, while it is simple to give (the  trivial) examples of graphs $G$ with $\chi(G) = r, \chi_D(G) = r+s$, for any $r,s$, non-trivial examples are a little harder to come by. Clearly, adding a copy (not necessarily disjoint) of a large complete multipartite graph to an arbitrary graph achieves this goal but such examples, we shall consider `trivial' since the reason for the blowing-up of the distinguishing chromatic number is  trivially attributed to the presence of the complete multipartite component.  While it seems simple to qualitatively ascribe the notion of what constitutes a nontrivial example in this context, we find it somewhat tedious to describe it precisely.  Our last result in this paper describes what we would like to believe constitutes a nontrivial family of bipartite graphs $G$ such that $\chi_D(G)>r+s$, for any  $l,k\ge 2$. It turns out that large complete bipartite graphs do appear as  induced subgraphs in our examples, but that alone does not guarantee that the distinguishing chromatic number necessarily increases. Furthermore,  what makes these nontrivial in our opinion, is the fact that the distinguishing chromatic number of these graphs  is more than what one might initially guess. 

The rest of the paper is organized as follows. In section 2, we state and prove what we regard as a variant of the motion lemma and use this to establish instances of families of graphs with $\chi_D(G)=\chi(G)+1$ in section 3. In section 4, we describe two families of graphs - $\G_1$ and $\G_2$ with rather contrasting properties. For $G\in\G_1$, we have $\chi_D(G)=\chi(G)+1$ even though $|Aut(G)|=O(|V|^{3/2})$; for $G\in\G_2$,  $|Aut(G)|=\omega(e^{|V|})$ and yet every proper $\chi(G)$ coloring of $G$ is in fact distinguishing. In section 5, we describe a family of bipartite graphs for which $\chi_D(G)>r+s$, for any $r,s\ge 2$. The last section contains some concluding remarks and open questions.

\section{A Variant of the Motion Lemma}\label{motionlemma}
Following \cite{RS}, we recall that the motion of an automorphism $\phi \in Aut(G)$ is  defined as $$m(\phi) := \{v \in G:\phi(v) \neq v \}$$ and the motion of a graph $G$ is defined as $$m(G) := \mathop{\min}\limits_{\substack{\phi \in Aut(G) \\ \phi\ne I}} m(\phi).$$ The Motion lemma of \cite{RS} states that for a graph $G$, if $m(G) > 2 \log_2|Aut(G)|$ then $G$ is $2-$distinguishable. We prove a slightly more general criterion to obtain a similar conclusion for the distinguishing chromatic number. 

For a graph $G$ with full automorphism group $Aut(G)$, let $\G\subset Aut(G)$ be a subgroup of the automorphism group.
For $A \in \G$ and $S \subseteq V(G)$ we define $Fix_A(S) =\{ v \in S : A(v) = v\}$ and $F_A(S) = |Fix_A(S)|$. Let $F(S):=\displaystyle\max_{\substack{A\in\G\\ A\ne I}} F_A(S)$. 

\begin{defn}The \textit{Orbit of a vertex $v$ with respect to an automorphism $A$} is the set $$Orb_A(v) := \{p,Av, A^2v, \hdots A^{k -1 }v\}$$ where $A^kv = v.$\end{defn}

\begin{lem}[A variant of the motion lemma]\hfill \break\label{motionlemma}
Let $C$ be a proper coloring of the graph $G$ with $\chi(G)$ colors and let $C_1$ be a color class in $C$. Let $\G$ be the subgroup of $Aut(G)$ consisting of all automorphisms that fix the color class $C_1$. For each $A\in\G$, let $\theta_A$ denote the total number of distinct orbits induced by the automorphism $A$ in the color class $C_1$. If 
$$\sum_{A \in \G } t^{\theta_A-|C_1| }< r$$ where $r$ is the least prime dividing $|\G|$,
for some integer $t\geq 2$, then $\chi_D(G) \leq \chi(G) + t-1$. In particular, if $F(C_1) < |C_1| - 2 \log_t|\G|$ then this conclusion holds.
\end{lem}
\begin{proof}
Let $1$ be the color assigned in the color class $C_1$. For each $v \in C_1,$ pick uniformly and independently, an element in $\{1, 2, \hdots, t \}$ and color $v$ using that color. Keep the labeling of all other vertices intact. This creates $t-1$ additional color classes. This new coloring $C'$ of $G$ is clearly proper; we claim that with positive probability, it is also distinguishing. 

For $A \in \G$, let $B_A$ denote the event that $A$ fixes every color class.  
Observe that if $A$ fixes a color class containing a vertex $v,$ then all other vertices in the set $orb_A(v)$ are also in the same color class. Moreover the probability that $Orb_A(v)$ is in the same color class of $v,$ equals $t^{1-{| Orb_A(v)|}}$. 
Then
$$ \bP(B_A) = \prod_{\theta_A}t^{1-| Orb_A(v)|} = t^{\theta_A -|C_1|}$$

Let $\N\subset\G$ denote the set of all automorphisms which fixes every color classes in $C'$ and let $N=|\N|$. Then note that
\begin{equation}\label{preqn}
\bE(N) = \sum_{A \in \G} \frac{1}{t^{|C_1| - \theta_A}}
\end{equation}
 If $\bE(N) < r$ then there exist a $\chi(G) + t-1$ proper coloring of $G$ satisfying $N<r$. Since $\N$ is in fact a subgroup of $\G$, $N$ divides $|\G|$, so if $\bE(N) < r$ it follows then that with positive probability, $\N=\{I\}$, so the coloring $C'$ is  distinguishing. 

In particular, since $\theta_A \leq F(C_1) + \frac{|C_1|-F(C_1)}{2}$ it follows from  equation (\ref{preqn}) that
\begin{equation*}
\bE(N) \leq \sum_{A \in \G} t^{\frac{F(C_1)-|C_1|}{2}}= |\G|t^{\frac{F(C_1)-|C_1|}{2}}.
\end{equation*}
Thus, if $F(C_1) < |C_1| - 2 \log_t|\G|$ then there exist a distinguishing proper $\chi(G) + t-1$ coloring of the graph.
\end{proof}
   
   \section{Examples} 
\subsection{Levi graphs}
In this subsection, we restrict our attention to Desarguesian projective planes and consider the Levi graphs of these projective planes, which are the bipartite incidence graphs corresponding to the set of points and lines of the projective plane. It is well known \cite{huges} that the theorem of Desargues is valid in a projective plane if and only if the plane can be constructed from a three dimensional vector space over a skew field, which in the finite case reduces to the three dimensional vector spaces over finite fields.

In order to describe the graphs we are interested in, we set up some notation. Let $\bF_q$ denote the finite field of order $q$, and let us denote the vector space $\bF_q^3$ over  $\bF_q$ by $V$.  Let $\Pa$ be the set of $1$-dimensional subspaces of $V$ and $\La$, the set of $2$-dimensional subspaces of $V$. We shall refer to the members of these sets by points and lines, respectively. The \textit{Levi graph} of order $q$, denoted by $LG_q$ is a bipartite graph defined as  follows: $V(LG_q) = \Pa\cup\La$, where this describes the partition of the vertex set; a point $p$ is adjacent to a line $l$ if and only if $p\in l$.

The Fundamental theorem of Projective Geometry \cite{huges} states that the full group of automorphisms of the  projective plane \PG   \hspace{0.05cm} is induced by the group of all non-singular semi-linear transformations $P\Gamma L(V)$ of $V$ onto $V,$ where $V$ is the corresponding vector space of \PG. If $q = p^n$ for a prime number $p$,   $P\Gamma L(V) \cong PGL(V) \rtimes Gal ( \frac{K}{k}),$ where $Gal (\frac{K}{k}) $ is the Galois group of $K:=\mathbb{F}_q$ over $k:=\mathbb{F}_p$. In particular, if  $q$ is a prime, we have $P\Gamma L(V) \cong PGL(V).$ The upshot is that $LG_q$ admits a large group of automorphisms, namely,  $P\Gamma L(V)$\footnote{It follows that this group is contained in the full automorphism group. The full group is  larger since it also includes  maps induced by isomorphism of the projective plane with its dual.}.

We first show that the distinguishing chromatic number for the Levi graphs $LG_q$ is precisely $3$ in almost all the cases. This is somewhat reminiscent of the result of \cite{contuck} for the distinguishing number of affine spaces. 
\begin{thm}
$\chi_D(LG_q) = 3$ for all prime powers $q \geq 5.$  
\end{thm}
\begin{proof}
Firstly, we consider the case when $q \geq 5$ and $q$ is prime and show that $\chi_D(LG_q) \leq 3.$ Consider a 2-coloring of $LG_q$ by assigning color $1$ to the point set $\mathcal{P}$ and color $2$ to the line set $\mathcal{L}$. It is easy to see that an automorphism of $LG_q$ that maps $\Pa$ into itself and $\La$ into itself corresponds to an automorphism of the underlying projective plane, and hence any such automorphism is necessarily in $PGL(V)$ (by the preceding remarks). 

In order to use lemma \ref{motionlemma},  observe that, any $I \neq A\in PGL(V)$ fixes at most $q+2$ points of $LG_q$.
Hence $$\theta_A\leq q+2+\frac{(q^2+q+1)-(q+2)}{2}=\frac{q^2 + 2q + 3}{2}.$$ Consequently,
\begin{equation}\label{eqnlevi5}
\bE(N) < \frac{(q^8 - q^6 - q^5 + q^3)}{2^{(q^2 + 1)/2}} +1
\end{equation}
{\bf Case 1:} $q \geq 7.$\\
For $q = 7, t = 2,$ the right  hand  side of  \ref{eqnlevi5} is approximately $1.3.$ Since the right hand side of  inequality \ref{eqnlevi5} is monotonically decreasing in $q$, it follows  that  $\bE(N)<2$ for $q \geq 7$, hence by lemma \ref{motionlemma} $LG_q$ admits a  proper distinguishing $3-$coloring. In particular,  $\chi_D(LG_q) = 3,$ for $ q \geq 7,$ since clearly, $\chi_D(LG_q)>2$.\\
{\bf Case 2:} $q = 5.$\\
In this case, for $t=2$ we actually calculate $\bE(N)$ using the open source Mathematics software SAGE to obtain $\bE(N) \approx 1.2.$ Therefore, again in this
case we have $\chi_D(LG_5) = 3.$

For $q=2$, it turns out that $\chi_D(LG_3)=4$ and for $q=3$ we are only able to prove $\chi_D(LG_5) \leq 5$. We include these proofs in the Appendix for the sake of completeness.

If $q=p^n$ for $n\ge 2$ and a prime $p$,  we note that the cardinality of the automorphism group of \PG\hspace{0.06cm} equals $$n|PGL(V)|\le \log_2(q)|PGL(V)|.$$
 
 As in the prime case, we have  $$\bE(N)
  \le \frac{\log_2 q(q^8 - q^6 - q^5 + q^3)}{t^\frac{q^2 + 1}{2}} + 1. $$
 
  For $q = 8$ and $t = 2$ the right hand side is approximately $1.01.$ By the same arguments as in the preceding section,  it follows that $\chi_D(LG_q) = 3.$ 
\end{proof}

For $q=4$ we calculate  $\bE(N) \approx 1.2.$ for $q=4,$ and $t=3$ using SAGE to make the actual computation, so we have  
 $\chi_D(LG_4) \leq 4$.  We believe that $\chi_D(LG_4) = 3$ though again, our methods fall short of proving this.

\subsection{Levi graphs of order one}
 Suppose $n,k\in\bN$ and $2k<n$, and consider the bipartite graphs  $G = G(L,R,E)$ where $L:=\binom{[n]}{k-1}$ corresponds to the  set of $k-1$ subsets of $[n]$, $R:=\binom{[n]}{k}$ corresponds to the $k$ subsets of $[n]$, and $u\in L, v\in R$ are adjacent if and only if $u\subset v$. We shall refer to these graphs as \textit{Levi Graphs of order $1$} and we shall denote them by   $LG_1(k,n)$, or sometimes, simply $LG_1$. Note that for each $u\in L, v\in R$ we have  $d(u) = n-k+1$ and $d(v) = k.$
 
It is clear that $S_n\subset Aut(LG_1)$. But in fact $Aut(LG_1) = S_n$, and this is a fairly routine exercise, so we skip these details.

 
We shall use lemma \ref{motionlemma} to determine the distinguishing chromatic number of $LG_1(k,n)$. Following the notation of the lemma, set $F_{\sigma} := \{ v \in R : \sigma(v) = v \}$ for  $\sigma \in S_n$   and let $F=\displaystyle\max_{\substack{\sigma\in S_n\\\sigma\ne I}} |F_{\sigma}|$. 

\begin{lem}\label{setfix}
For $n>4$,  $F \leq {n-2 \choose k-2} + {n-2 \choose k}$ and equality is attained if and only if $\sigma$ is a transposition $(ij)$ for some $i\ne j$. 
\end{lem}

\begin{proof} Firstly, it is easy to see that if $\sigma=(12)$ then $|F_{\sigma}|= {n-2 \choose k-2} + {n-2 \choose k}$, so it suffices to show that for any $\pi$ that is not of the above form, $|F_{\pi}|<|F_{\sigma}|$.

 Suppose not, i.e., suppose $\pi \in S_n$ is not an involution and $|F_{\pi}|$ is maximum. Write $\pi = O_1 O_2 \hdots O_t$ as a product of disjoint cycles with $|O_1|\ge|O_2|\ge\cdots\ge|O_t|$.  Then either $|O_1|>2$, or $|O_1| = |O_2| =2$. If $ |O_1| > 2$, then suppose without loss of generality, let $O_1 = (123\cdots)$
 If $h \in F_{\pi} $ then either $\{1,2\}
\subset h$ or $\{1,2\} \cap h = \emptyset.$ In either case we observe that  $h \in F_{\sigma}$ as well. Therefore $F_{\pi} \subset F_{\sigma}.$ Furthermore, note that $\sigma$ fixes the set $g = \{1, 2, 4, \hdots, k+1\},$ while $\pi$ does not. Hence $|F_{\sigma}| > |F_{\pi}|$, contradicting that $|F_{\pi}|$ is maximum.  If $|O_1|= |O_2|=2$, again  without loss of generality let $O_1= (12), O_2=(34).$ Again, $h\in F_{\pi}$ implies that either $\{1,2\} \subset h$ or $\{1,2\} \cap h =\emptyset,$ so once again, $h\in F_{\sigma}\Rightarrow h\in F_{\pi}$.  Furthermore, $\{1,2,3,5,\hdots,k+1\}\in F_{\sigma}\cap \overline{F_{\pi}},$ which contradicts the maximality of $|F_{\pi}|$.\end{proof}
For $k\ge 2$ define $n_0(k) := 2k+1$ for $k\ge 3$ and $n_0(2) := 6$.
\begin{thm}
 $\chi_D(LG_1(k,n)) = 3$ for $k\ge 2$ for $n\ge n_0(k)$. 
\end{thm}

\begin{proof}  We deal with the cases  $k =2, k=3 $ first, and then consider the general case of  $k>3.$

For $k=2$, let $A=\{(1,2), (2,3), (2,4), (3,4), (4,5), (5,6), \hdots, (n-1,n) \}$, and consider the coloring with the color classes being  $L, A, R \setminus A$. Consider the graph $G$ with $V(G) = [n]$ and $E(G) = A.$  Observe that the only automorphism $G$ admits is the identity. Since a nontrivial automorphism that preserves all the color classes of this coloring must in fact be a nontrivial automorphism of $G$, it follows that the coloring described is indeed distinguishing. If $k=3$, note that the color classes described by the sets $R,A, L\setminus A$ is proper and distinguishing for the very same reason.

For the case $k\ge 4$, we use lemma \ref{motionlemma} with $t=2$. From Lemma \ref{setfix} we have $F \leq {n-2 \choose k-2} + {n-2 \choose k}$. Let $C_1 = R$ be the color class to be parted randomly and assign color $3$ to all vertices in $L=\binom{[n]}{k-1}$. Then we have,
\begin{equation}\label{prlg1}
\bE(N) \leq |Aut(LG_1)|2^{\frac{1}{2}(F-|C_1|)}+1,
\end{equation}
 where $|C_1|= {n\choose k}$. 

Therefore from Equation \ref{prlg1}, we have 
$$ \bE(N)\leq \frac{n!}{2^K} +1,$$
where $K = \frac{{n \choose k} - {n-2 \choose k-2} - {n-2 \choose k}}{2}.$
For $n>2k$ it is not hard to show that $ \frac{n!}{2^K} < 1$ for $n\ge n_0(k)$, so, by lemma \ref{motionlemma} we are through.
\end{proof} 

\subsection{Weak product of Graphs}
The distinguishing chromatic number of a Cartesian product of graphs has been studied in \cite{Hemanshu}. The fact that any graph can be uniquely (upto a permutation of the factors) factorized into prime graphs with respect to the Cartesian product plays a pivotal role in determining the full automorphism group. In contrast, an analogous theorem for the weak product only holds under certain restrictions. In this subsection, we consider the $n$-fold weak product of certain graphs and consider the problem of determining their distinguishing chromatic number.
  
To recall the definition again, the weak product (or Direct product as it is sometimes called) of graphs $G,H$ denoted $G \times H$, is defined as follows: $V(G \times H) = V(G) \times V(H)$. Vertices $ (g_1, h_1), (g_2, h_2)$ are adjacent if and only if $\{g_1,g_2\}\in E(G)$ and $\{h_1, h_2\}\in E(H)$. We first collect a few basic results on the weak product of graphs following \cite{rws}. For more details we refer the interested reader to the aforementioned handbook.

 Define an equivalence relation $R$ on $V(G)$ by setting $xRy$ if and only if $N(x) = N(y)$ where $N(x)$ denotes the set of neighbors of $x$.  A graph $G$ is said to be $R-thin$ if each equivalence class of $R$ is a singleton, i.e., no distinct $x,y\in V(G)$ have the same set of neighbors. A graph $G$ is \textbf{prime with respect to the weak product}, or simply prime, if it is nontrivial and $G \cong G_1 \times G_2$ implies that either $G_1$ or $G_2$ equals $K_1^s$, where $K_1^s$ is a single vertex with a loop on it.  Observe that $K_1^s \times G \cong G$. 
 
Before we state our main theorem of this subsection, we state two useful results regarding the weak product of graphs. If $G$ is connected, nontrivial, and non-bipartite then the same holds for  $G^{\times n}$. This is a simple consequence of a theorem of Weischel (see  \cite{rws} for more details ). The other useful result is the following theorem which also appears in \cite{rws}.

 \begin{thm}\label{autweak}
 Suppose $\phi$ is an automorphism of a connected nonbipartite $R-thin$ graph $G$ that has a prime factorization $G \cong G_1 \times G_2 \times \hdots \times G_k$ . Then there exist a permutation $\pi$ of $\{1,2,\hdots, k\},$ together with isomorphisms $\phi_i: G_{\pi(i)} \rightarrow G_i,$ such that $$\phi(x_1, x_2, \hdots, x_k) = (\phi_1(x_{\pi(1)}), \phi_2(x_{\pi(2)}), \hdots, \phi_k(x_{\pi(k)})).$$ 
 \end{thm}

We are now in a position to state our main result regarding the distinguishing chromatic number for a weak product of prime graphs. An analogous result for the cartesian product of graphs, under milder assumptions,  appears in \cite{Hemanshu}.
 \begin{thm}\label{thmweak}
 Let $G$ be a connected, nonbipartite, $R-thin$, prime graph on at least $3$ vertices. Denote by $G^{\times n}$ the $n$-fold product of $G$, i.e., $G^{\times n}:=\overbrace{G \times G \times \hdots \times G}^{n\textrm{-times}}$.  Suppose further that $G$ admits a proper $\chi(G)$ coloring with a color class $C_1$ such that no non-trivial automorphism of $G$ fixes every vertex of $C_1$. Then $\chi_D(G^{\times n} ) \leq \chi(G) + 1$ for $n\ge 4$.
\end{thm}
\begin{proof} 
 Let $G$ be connected, non-bipartite, $R-thin$, and prime. We first claim that 
 $$Aut(G^{\times n}) \cong Aut(G) \wr S_n, $$ the wreath product of $Aut (G)$ and $S_n$.
To see this, note that if $G$ is $R-thin$ one can easily check that $G^{\times n}$ is also $R-thin$. Moreover since every connected non-bipartite nontrivial graph  admits a unique prime factorization for the weak product (see \cite{rws}),  it is a simple application of theorem \ref{autweak} to see that $Aut(G^{\times n}) \cong Aut(G) \wr S_n $. This proves the claim.
 
Suppose $\chi(G) = r$ and let $\{C_i: i \in [r]\}$ be a proper coloring of $G$. Then $C_i \times G^{\times n-1}, i\in [r]$ is a  proper $r$ coloring of the graph $G^{\times n}$, so $\chi(G^{\times n}) \leq r$. On the other hand, the map $g\to(g,g\hdots,g)$ is a graph embedding of $G$ in $G^{\times n}$, so $\chi(G^{\times n})=r$. 
Let us denote the aforementioned color classes of $G^{\times n}$ by $C'_i, i \in [r]$. We claim that $\chi_D(G^{\times n}) \leq r+1$ and show this as a consequence of
lemma \ref{motionlemma}.

By hypothesis there exist a color class, say $C_1$ in $G$ such that no nontrivial automorphism fixes each $v\in C_1$.   Consider $C'_1 = C_1 \times G^{\times n-1}$ and for each element in $C'_1$ assign a value from $\{1,r+1\}$ uniformly and independently at random. This describes  a proper $(r+1)-$coloring of $G^{\times n}$. By lemma \ref{motionlemma}, we have  
\begin{equation}\label{eqnweak}
\bE(N) \leq n!|Aut(G)|^n 2^{\frac{F-T}{2}} + 1
\end{equation}
where $T=|C_1 \times G^{\times n-1}|$ and $F$ is as in lemma \ref{motionlemma}. 

 \textbf{Claim:} If there exists a nontrivial automorphism of $G^{\times n}$ which fixes each  color class $C'_i, i=1\ldots, r$, then it cannot also fix each vertex of  $C_1'$. 
 
 To prove the claim, suppose $\psi$ is an automorphism go $G^{\times n}$ which fixes $C'_i$ for each $i\in[r]$, and also fixes $C'_1$ point-wise. By theorem \ref{autweak}, there exist $\phi_1, \phi_2, \hdots , \phi_n\in Aut(G)$ and $\pi \in S_n$ such that 
\begin{equation}\label{eqnpsi}
\psi(x_1, x_2, \hdots, x_n) = (\phi_1(x_{\pi(1)}), \phi_2(x_{\pi(2)}), \hdots, \phi_n(x_{\pi(n)}))
\end{equation} for all $(x_1,x_2,\ldots,x_n)\in G^{\times n}$.
Now note that if $\psi$ fixes $C'_1$ point-wise then $\phi_1$ fixes $C_1$ point-wise. Indeed, 
\begin{eqnarray}
\psi(x_1, x_2, \hdots, x_n) &=& (x_1, x_2, \hdots, x_n)\nonumber\\
\iff (\phi_1(x_{\pi(1)}), \phi_2(x_{\pi(2)}), \hdots, \phi_n(x_{\pi(n)})) &=& (x_1, x_2, \hdots, x_n) \nonumber \\
\iff \phi_i(x_{\pi(i)})&=& x_i \quad \textrm{\ for\ all\ } i\in[r].\label{eqnfix}
\end{eqnarray} 
Since \ref{eqnfix} holds for all vertices $(x_1, x_2, \hdots, x_n)\in G^{\times n}$ with $x_1 \in C_1$ and $x_i \in G$ , $2 \leq i \leq n$, we conclude that $\pi=I$, $\phi_i=I$, for  $2\leq i \leq n$, and $\phi_1$ acts trivially on  $C_1$. But then by the hypothesis on $G$, it follows that $\phi_1= I$ in $G$ and hence $\psi=I$. 

 We now show that  $F \le (|C_1| -2)|G|^{n-1}.$
 
We adopt similar notations as in Lemma \ref{motionlemma} and for simplicity, let us denote $|G|=m$.
For $\psi \in Aut(G^{\times n})$ we shall write $\psi = (\phi_1, \phi_2, \hdots, \phi_n: \pi)$ to denote the map $$\psi( x_1, x_2, \hdots, x_n) = (\phi_1(x_{\pi(1)}), \phi_2(x_{\pi(2)}), \hdots, \phi_n(x_{\pi(n)}))$$ as in equation \ref{eqnpsi} (see theorem \ref{autweak}). Suppose $\psi$ fixes the vertex $( x_1, x_2, \hdots, x_n) \in G^{\times n}$. In particular
we have  $x_{\pi(i)} = \phi_i^{-1}(x_i)$ for all $i$. It then follows that for all $k$, we have $$\phi^{-1}_{\pi^k(i)}\left(x_{\pi^k(i)}\right)=x_{\pi^{k+1}(i)}$$ for each $i$. Consequently, if $\pi$ has $t$ cycles in its disjoint cycle representation then  $\psi$ can fix at most $|C_1|m^{t-1}$ vertices in $C'_1$. 

If $\pi\ne I$, then $t<n$, and in this case, since $m\ge 3, n\ge 4$, we have $|C_1|m^{t-1}\le (|C_1| -2)m^{n-1}$.  If $\pi=I$, then $\psi$ is non-trivial if and only if $\phi_i\ne I$ for some $i$. In this case  $\phi_i(x_i) = x_i$ for all $i$, so  $(x_1,x_2,\ldots,x_n)$ is fixed by $\psi$ if and only if $x_i \in Fix_{\phi_i}$ for all $i$. Consequently,  $F_{\psi'} = \prod\limits_{i=1}^n F_{\phi_i}$.  Observe that if $\phi_i$ is not a transposition then it moves at least three vertices, say $x$, $y$ and $z$ in $G$. In particular,  $\psi$ does not fix any vertex of the form $(x_1, x_2, \hdots, g, \hdots, x_n)$ , where $g \in \{x,y,z\}$ and appears in the $i^{th}$ position.  Thus, it follows that
$$F_{\psi} \leq |C_1|m^{n-2}(m-3).$$  If $\phi_i$ is a transposition for some $i>1$ then it is easy to see that  $F_{\psi} \leq (|C_1| -3)m^{n-1}<(|C_1| -2)m^{n-1}$. Finally, if $\phi_1$ is a transposition, then again $F \le (|C_1| -2)m^{n-1}.$ This proves the claim.
 
Setting $F = (|C_1| -2)m^{n-1}$ , $T = |C_1|m^{n-1} $ in  equation (\ref{eqnweak}) gives us
$$\bE(N) \leq \frac{ n!|Aut(G)|^n}{2^{m^{n-1}}} + 1.$$
It is a simple calculation to see that the first term in the above expression is less than $1$ for all $m\ge 3$ and $n\ge 4$. This completes the proof.
\end{proof} 
 
\begin{cor}
$\chi_D(K_r^{\times n}) = r+1$ for $n\ge 4$, and $r\ge 3$.  
\end{cor}
\begin{proof} First note that for $r\ge 3$,  $K_r$ is prime, non-bipartite, and $R-thin$. Hence by theorem \ref{thmweak} it follows that $\chi_D(K_r^{\times n})\le r+1$. A result of Greenwell and Lov\'asz (see  \cite{greenwell}) tells us that all proper $r-$colorings for $K_r^{\times n}$ are induced by colorings of the factors $K_r$. In particular, it implies that $\chi_D(K_r^{\times n})>r$.   
\end{proof}

\section{$\chi_D(G)$ versus $|Aut(G)|$}
As indicated in the introduction, one aspect of the problem of the distinguishing chromatic number of particular interest is the contrasting behavior of the distinguishing chromatic number vis-\'a-vis the size of the automorphism group. Our sense of contrast here is to describe  the size of the automorphism group as a function of the order of the graph. 

We first show examples of graphs that admit  `small' automorphism groups, and yet have $\chi_D(G)>\chi(G)$ and with arbitrarily large values of $\chi(G)$. To describe these 
graphs,  let $q\ge 3$ be prime and suppose $S \subset \bF_q$ is a subset of size $\frac{q-1}{2}$. The graph $G_S$ is defined as follows:  $V(G_S) = \bF_q^2$ and vertices $u =(u_1,u_2)$, $v= (v_1,v_2)$ are adjacent if and only if $v_1\neq u_1$ and $(v_2-u_2)(v_1-u_1)^{-1} \in S$. We denote $(v_2-u_2)(v_1-u_1)^{-1}$ by $s(u,v)$. Observe that, for each $\alpha \in S$ and $\beta\in\bF_q$, the set $l_{\alpha}^{\beta}:=\{(\beta+x, \beta+x\alpha): x\in\bF_q\}$ is a clique of size $q$, so $\chi(G_S)\geq q$. We shall denote the independent sets \footnote{These are independent in $G_S$ since $\infty\notin\bF_q$.} $\{(\beta, x+\beta): x\in\bF_q\}$ by $l^{\beta}_{\infty}$. Similarly, if $\alpha \notin S$ the set $l_{\alpha}^{\beta}$ is an independent set of size $q$, the collection $\{l_{\alpha}^{\beta}: \beta \in \bF_q\}$ describes a proper $q$-coloring of $G_S$,  hence $\chi(G_S)= q$.  We shall call the sets $l_{\alpha}^{\beta}$ as \textit{lines} in what follows.
 
 \begin{thm}
  $\chi_D(G_S) > q$.  
 \end{thm}
  \begin{proof}
Let $C= \{C_1, C_2, \hdots, C_q\}$ be a proper $q-$coloring of $G_S$. We claim that each $C_i$ is a line, i.e., for each $1\le i\le q$ we have $C_i=l_{\alpha}^{\beta}$ for some $\alpha\notin S, \beta\in\bF_q$.

 Observe that for $\alpha \in S$, the collection $ \C=\{l_{\alpha}^{\beta}|\beta\in\bF_q\}$ partitions the vertex set of $G_S$ into cliques of size $q$. Therefore in any proper $q$-coloring of $G_S$ each color class contains exactly $q$ vertices. By a result by  Lov\'{a}sz and  Schrijver \cite{lovas},  for a prime number $q$ if $X\subset AG(2,q)$ such that $|X| = q$ and $X$ is not a line then the set  
 $S(X)=\{s(x,y)|x\ne y, x,y\in X\}$ \textrm{\ has\ size\ at\ least\ }$\frac{q+3}{2}.$
 If a color class $C_i$ is not a line then $|S(C_i)|\ge\frac{(q+3)}{2}$ and since $|S| = \frac{q-1}{2}$ this implies that $S(C_i)\cap S\ne\emptyset$. But then this contradicts that $C_i$ is independent in $G_S$. 

In particular, any proper $q$-coloring $C$ of $G_S$ must be a partition of the form $\{l_{\alpha}^{\beta}:\beta\in\bF_q\}$ with $\alpha\in(\bF_q\cup\{\infty\})\setminus S$. Then the map
\begin{eqnarray*} \phi_{\alpha}(x,y) &=& (x+1,y+\alpha)\hspace{0.5cm}\textrm{\ if\ }\alpha\neq\infty,\\ \phi_{\infty}(x,y)&=&(x,y+1)\end{eqnarray*} is a nontrivial automorphism that fixes each color class of $C$. This establishes that $\chi_D(G_S) > q$.
\end{proof}
Now, we shall show that for a suitable choice of $S$,  $G_S$ has a relatively small automorphism group, and for such $S$, we shall also show that $\chi_D(G_S)=q+1$. Our choice of subset $S$ shall be a uniformly  random subset of $\bF_q$. 

Observe that if $\phi \in Aut(G_S)$ then since maximum cliques (respectively, maximum independent sets) are mapped into maximum cliques (resp. maximum independent sets), it follows that $\phi$ is a bijective map on $\bF^2_q$ which maps affine lines  into affine lines in $AG(2,q)$ (as a consequence of \cite{lovas}). Hence, it follows that 
$Aut(G_S) \subset AGL(2,q)$ (see \cite{huges}). In other words, any $\phi \in Aut(G_S)$ can be written as $A + \bar{b}$ for some $A \in Aut_0(G_S)$ and $\bar{b} (=\phi(0,0)) \in \bF^2_q$, where $Aut_0(G_S)\subset Aut(G_S)$ is the subgroup of automorphisms which fix the vertex $(0,0) \in V(G_S)$.

\begin{thm}\label{autgs}
Suppose $S$ is picked uniformly at random from the set of all $\frac{q-1}{2}$ subsets of $\bF_q$. Then asymptotically almost surely, $Aut(G_S) = \{ \lambda I + \bar{b}: \lambda \in \bF_q^*,\quad \bar{b} \in V(G_S)\}$. Consequently, $|Aut(G_S)| = q^2(q-1)$ asymptotically almost surely.
\end{thm}
Here by the phrase asymptotically almost surely we mean that  the probability that $Aut(G_S) = \{ \lambda I + \bar{b}: \lambda \in \bF_q^*,\bar{b} \in V(G_S)\}$ approaches $1$ as $q\to\infty$.
\begin{proof}
Since we have already observed that $Aut(G_S)\subset AGL(2,q)$, every $\phi\in Aut(G_S)$ can be written in the form   $\phi(x,y) = A(x,y) + (b_1, b_2)$ for some $b_1, b_2 \in \bF_q$ and $A \in Aut_0(G_S)$. Here, $A\in GL(2,q)$ corresponds to a matrix
 $ \begin{pmatrix}
a & b \\ c & d
\end{pmatrix} $  
for $a,b,c,d \in \bF_q$ with  $ad-bc\ne 0$.

We introduce the symbol $\infty$ and adopt the convention that $a+\infty=\infty, a\cdot\infty=\infty$ for $a\ne 0$, and $\frac{a}{0}=\infty$ for $a\ne 0$. For $\phi \in Aut(G_S)$, define a map $f_\phi:\bF_q\cup\{\infty\}\to\bF_q\cup\{\infty\}$ as follows:
\begin{eqnarray*}
f_{\phi}(\alpha)&= &\frac{d \alpha +c}{a+b\alpha},\hspace{0.51cm}\textrm{\ if\ }\alpha\ne-\frac{a}{b},\\
f_{\phi}\left(\frac{-a}{b}\right)&=& \infty,\\
f_{\phi}(\infty)&=&\frac{d}{b}.\end{eqnarray*}

Observe that $f_\phi$ is trivial if and only if $b=c=0$ and $a=d$. 

Let $x = (x_1, x_2), y = (y_1, y_2)$ be two adjacent vertices in $G_S$.  Since $\phi(x)$ is adjacent to $\phi(y)$, we have $$s(\phi(x), \phi(y)) = \frac{ c(y_1 - x_1) + d(y_2 - x_2)}{a(y_1 - x_1) + b(y_2 - x_2)}=  \frac{  d\cdot s(x,y)+c }{ b\cdot s(x,y)+a}.$$ Observe that $y_1 - x_1$ is nonzero since $s(x, y) \in S$.  Therefore we have, 
\begin{equation}\label{eqnf}
s(\phi(x), \phi(y))  = f_{\phi}(s(x,y)).
\end{equation}
Also note that if $\phi\ne \lambda I$ then for $\mu \in \bF_q$ and $k \in \bN$, setting 
$$f^{(k)}_{\phi}(\mu):=\underbrace{f_{\phi}\circ f_{\phi}\circ\cdots\circ f_{\phi}}_{k-\textrm{fold}}(\mu) = \mu$$ yields a quadratic equation in $\mu$, so there are at most two values of $\mu\in\bF_q$ satisfying $f^{(k)}_{\phi}(\mu)=\mu$. In other words, for each positive integer $k$, the map $f_{\phi}$ admits at most two orbits of size $k$. Moreover if $A \in Aut(G_S)$ then by equation \ref{eqnf}, $f_A(S) = S.$

Consider the event $E$: There exist a nontrivial automorphism $A \in Aut_0(G_S)$ such that $f_{A}$ is not the identity map.
 Observe that $E$ is the union of the events $E_A$ where the event $E_A$ is described as follows: For any $A\in GL(2,q)$, $A\neq \lambda I$ for some $\lambda\ne 0$, $S$ is the union of $f_{A}$ orbits. Recall that  $f_{A}$ is not the trivial map if and only if $A\ne\lambda I$ for some $\lambda\ne 0$.  
 
 By a favorable automorphism, we shall mean an automorphism  $A\in Aut_0(G_S)$, $A\ne\lambda I$ such that $S$ is union of $f_A$ orbits.  By the preceding discussion, it follows that a favorable automorphism of $G_S$ induces a partition $\Lambda$ of $\frac{q-1}{2}$ in which there are at most two parts of any size. Therefore the number of favorable automorphisms is at most twice the number of integer partition of $\frac{q-1}{2}$ in which there are at most two parts of any size which is clearly less than  $2p(\frac{q-1}{2})$, where $p(n)$ denotes the partition function. By the asymptotics of the partition function of Hardy-Ramanujan (see \cite{rama}),  $$p(t)\sim\frac{1}{4t \sqrt{3}}\exp\Bigg(\pi \sqrt{\frac{2t}{3}}\Bigg),$$ where $t=(q-1)/2$. So in particular,  for any $A \in Aut_0(G_S)$ the probability that $f_A$ is nontrivial is less than $p(t){q\choose t}^{-1}$. Consequently,  $$\bP(E) \leq (q^2 -1)(q^2-q) \frac{2p(t)}{{q\choose t}}\to 0\textrm{\ as\ } q \rightarrow \infty.$$ Hence asymptotically almost surely,  every $S \subset \bF_q$ satisfies $Aut_0(G_S) = \{\lambda I : \lambda \in \bF_q \}$. The second statement follows trivially from this conclusion.        
\end{proof}
In particular, let $S$ be a subset of $\bF_q$ such that $Aut_0(G_S) = \{\lambda I : \lambda \in \bF_q \}$. For such $S$, the distinguishing chromatic number of $G_S$ is greater than its chromatic number.
\begin{thm}Let $S\subset\bF_q$ be a set of size $\frac{q-1}{2}$ such that $Aut_0(G_S) = \{\lambda I : \lambda \in \bF_q \}$. Then 
$\chi_D(G_S) = q+1$.
\end{thm}
\begin{proof}
For $1 \neq \gamma \notin S$, consider the coloring of $G_S$ described by the color classes $\{l_{\gamma}^{\beta}: \beta \in \bF_q\}$. Assign the color $q+1$ to only the vertex $(0,0) \in V(G_S)$. This forms a $q+1$ coloring of $G_S$which is  obviously a proper coloring. To show that this is distinguishing,  let $\phi$ be a color fixing nontrivial automorphism of $G_S$. By theorem \ref{autgs}, $\phi$ maps $(x,y)$ to $(ax+b_1, ay+b_2)$ for some $a,b_1,b_2 \in \bF_q$. Since $\phi$ fixes $(0,0)$ we have $b_1=b_2=0$ and $a \neq 1$. This implies $\phi = aI$ and hence it is not color fixing; indeed $\phi$ maps $(1,1)$ to $(a,a)$ and $(a,a) \notin l_{\gamma}^{1}$. 
\end{proof}
 
 Our second result in this section describes a family of graphs with very large automorphism groups - much larger than exponential in $|V(G)|$, but for which $\chi_D(G)=\chi(G)$. As was proven in \cite{kneser}, we already know that the Kneser graphs $K(n,r)$ with $r\ge 3$ satisfy the same. However, one might also expect that in such cases, distinguishing proper colorings are perhaps rare, or at the very least, that there do exist minimal proper, non-distinguishing colorings of $G$.  It turns out that even this is not true.

\begin{thm}Let $\overline{K(n,r)}$  denote the complement of the Kneser graph, i.e., the vertices of $\overline{K(n,r)}$ correspond to $r$ element subsets of $[n]$ and two vertices are adjacent if and only if their intersection is non-empty. Then for $n\ge 2r$ and $r\ge 3$
$\chi_D(\overline{K(n,r)})= \chi(\overline{K(n,r)})$. Moreover, every proper coloring of $\overline{K(n,r)}$ is in fact distinguishing.
\end{thm} 
\begin{proof} 
First,  observe that since $Aut(K(n,r))\simeq S_n$ for $n\ge 2r$, the full automorphism group of $\overline{K(n,r)}$ is also $S_n$. 

Consider a proper coloring $c$ of $\overline{K(n,r)}$ into color classes $C_1,C_2,\ldots, C_t$. Note that for any two vertices $v_1,v_2$ in the same color class, $v_1\cap v_2=\emptyset$. If possible, let $\sigma\in S_n$ be a non-trivial automorphism which fixes $C_i$ for each $i$. Without loss of generality let $\sigma(1)=2.$  Observe that for the vertex $v_1 = (1,2,\hdots, r),$ its color class has no other vertex containing $1$ or $2,$ so  $\sigma$ maps $\{1,2,,\hdots, r\}$ to $\{1,2,\hdots, r\}.$ Again, with the vertex $v_2 = \{1,3,\hdots, r+1\}$, which is in color class $C_2\neq C_1$, $\sigma$ maps $v_2$ into $\{2, \sigma(3), \hdots, \sigma(r+1)\}\neq v_2,$ so $\sigma(v_2)\cap v_2=\emptyset$ by assumption.  However, since $\sigma(i)\in\{1,2,\ldots,r\}$ for each $3\le i\le r$ this yields a  contradiction.
\end{proof}

\section{Bipartite Graphs with large $\chi_D(G)$}\label{lgq}

In this section we describe a family of bipartite graphs whose distinguishing chromatic number is greater than $r+s$, for any $r,s\ge 2$. As we described in the introduction, the sense of non-triviality of these examples arises from a couple of factors. Our examples contain several copies of $K_{r,s}$ as induced subgraphs. That by itself does not imply that the distinguishing chromatic number is at least $r+s$ but it is suggestive. What makes these families nontrivial is the fact that the distinguishing chromatic number of these graphs is in fact $r+s+1$. 

Again, in order to describe these graphs, let $q\ge 5$ be a prime power, and let $\Pi:=(\Pa,\La)$ be a Desarguesian projective plane of order $q$. As is customary, we denote by $[r]$, the set $\{1,2\ldots,r\}$.  

The graph which we denote  $LG_q \otimes K_{r,s}$ has vertex set $V(LG_q \otimes K_{r,s})=( \Pa \times[r])\sqcup(\La\times [s]) $, and for $p\in\Pa, l\in\La, \textrm{\ and\ }(i,j)\in[r]\times [s]$ we have $(p,i)$ adjacent to $(l,j)$ if and only if $p\in l$. Another way to describe this graph goes as follows.  The weak product $LG_q\times K_{r,s}$ is bipartite and consists of two isomorphic bipartite components. The graph $LG_q \otimes K_{r,s}$ is one of the connected components.

 For each point $p$ there are $r$ copies of $p$ in the graph $LG_q \otimes K_{r,s}$; we call the set $\{(p,i)|i\in[r]\}$ the fiber of $p$, and denoted it by $F(p)$. Similarly we denote by $F(l)$, the set $F(l) = \{(l,i): i \in [s]\}$, and shall call this the fiber of $l$. Each vertex $(p,i)$ (resp. $(l,j)$) of $LG_q \otimes K_{r,s}$ has degree $r(q+1)$ (resp. $s(q+1)$).  
\begin{thm}
 $\chi_D(LG_q \otimes K_{r,s}) = r+s + 1$, where $ r,s \geq 2$ and $q \geq 5$ is a prime number.
\end{thm} 

\begin{proof} Firstly, we show that $\chi_D(LG_q \otimes K_{r,s}) > r+s $.

If possible,  let $C$ be an $(r+s)$-proper distinguishing coloring of $LG_q \otimes K_{r,s}$ and let $C_i, i\in [r+s]$ be the color classes of $C$ in $LG_q \otimes K_{r,s}$. We claim:\begin{itemize}
\item[1.] For each $p \in \Pa$, each vertex of $F(p)$ gets a distinct color. The same also holds for each $l\in \La$ and each vertex of $F(l)$.
\item[2.] If $\C_{\Pa}$ and $\C_{\La}$ denote the sets of colors on the vertices of  $\bigcup\limits_{p \in \Pa} F(p)$ and  $\bigcup\limits_{l \in \La} F(l)$ respectively, then $\C_{\Pa}\cap\C_{\La}=\emptyset$ and $|\C_{\Pa}|=r,|\C_{\La}|=s$. Consequently, for each $i$,  either $F(p)\cap C_i\ne\emptyset$ for each $p\in\Pa$ or $F(l)\cap C_i\ne\emptyset$ for each $l\in\La$.
\end{itemize} 
We shall first prove each of the claims made above.
\begin{itemize}
\item[1.] For $p\in\Pa$ suppose $F(p)$ contains two elements, say $(p,i)$ and $(p,j)$, with the same color. Consider the map $\phi$ that swaps $(p,i)$ with $(p,j)$ and fixes all other vertices. It is easy to see that $\phi$ is a graph automorphism which fixes each color class $C_i$ contradicting  the assumption that $C$ is distinguishing. The argument for the part regarding vertices in the fiber $F(l)$ is identical.  
\item[2.] Let $l\in\La$ and $p\in l$. By claim $1$ each vertex in $F(p)$ has a distinct color. Since $|F(p)| = r$ we may assume without loss of generality let $(p,i)$ gets color $i$ for $i \in [r]$. In that case, no vertex of $F(l)$ can be colored using any color in $[r]$. Furthermore, by the same reasoning as above, each vertex of $F(l)$ is colored using  a distinct color, so we may assume again that $(l,i)$ is colored $r+i$ for $i=1,2\ldots,s$. Since there is a unique line through any two points, no vertex of the form $(p',j)$ gets a color in $\mathop{\cup}\limits_{i=r+1}^{r+s}C_i$. Similarly, no vertex of the form $(l',j)$  belongs to $\mathop{\cup}\limits_{i=1}^{r}C_i$.  Therefore, all points and their fibers belongs to $\mathop{\cup}\limits_{i=1}^r C_i$ and all lines with their fibers belongs to $\mathop{\cup}\limits_{i=r+1}^{r+s} C_i$. 

From claims 1 and 2 above, we conclude that for each $p \in \Pa$, $C_i \cap F(p) \neq \emptyset$ for $i\in[r]$. Otherwise, since $|F(p)| = r$, there exist an $i\in [r]$ such that $|C_i \cap F(p) | \geq 2$, contradicting  claim 1. Similar arguments show that for each $l \in \La$, $C_{i+r} \cap F(l) \neq \emptyset$ for $i\in[s]$.  
\end{itemize}
 To show $C$ is not a distinguishing coloring we produce a nontrivial automorphism of $LG_q \otimes K_{r,s}$ which fixes each $C_i$ for $i=1,2,\ldots, r+s$.  We first set up some terminology. For $i\in[r]$, we call a vertex in the fiber of $p$ its $i^{th}$ vertex if its color is $i$ and shall denote it $p^i$. Similarly, we shall call a vertex in the fiber of $l$ its $i^{th}$ point if its color is $i+r$ and shall denote it by $l^i$.
 
Let $\psi\in Aut(LG_q)$ be a nontrivial automorphism such that $\psi(\Pa) = \Pa$ so that it also satisfies $\psi(\La)= \La$. Let $\sigma$ be defined on $V(LG_q \otimes K_{r,s})$ by $\sigma(v^i) = \psi(v)^i$ for $v\in \Pa\sqcup\La$.  It is clear that $\sigma$ is a color preserving map. Moreover $\sigma$ preserves adjacency in $LG_q \otimes K_{r,s}$; indeed, $v$ is adjacent to $w$ in $LG_q$ if and only if $F(v)\cup F(w)$ forms a $K_{r,s}$ as a subgraph of $LG_q \otimes K_{r,s}$ and $\psi \in Aut(LG_q)$. Therefore $\sigma$ is a nontrivial automorphism which fixes the color classes, thereby showing that $\chi_D(LG_q \otimes K_{r,s}) > r+s.$

We now claim that  $\chi_D(LG_q \otimes K_{r,s}) \leq r+s+1 $. For $1 \leq i \leq r-1$, assign the color $i$ to the points $\{(p,i): p\in \Pa\}$ and for $r+1 \leq j \leq r+s$ let $\{(l,j): l\in \La\}$ be colored $j$. Recall that $LG_q$ admits a distinguishing $3$-coloring in which every vertex of $\La$ is given the same color, and the point set $\Pa$ is partitioned into  $\Pa_1,\Pa_2$ that correspond to the other two color classes (theorem \ref{lg3}). We split the set $\{(p,r): p\in \Pa\}$ into $C_r:=\{(p,r)|p\in\Pa_1\}$ and $C_{r+s+1}:=\{(p,r): p\in\Pa_2\}$ and designate  these sets as color classes $r$ and $r+s+1$ respectively.  

It is easy to see that the above coloring is proper since adjacent vertices get different colors. To see that it is distinguishing,  let $\mu$ be a nontrivial automorphism which fixes each color class. Since $\mu$ fixes each color class as a set, and $\mu$ is nontrivial, in particular, $\mu$ fixes the set $\{(p,r):p\in\Pa\}$, and also fixes each set $\{(l,i): l\in\La\}$ for $r+1\le i\le r+s$, so in particular, $\mu$ induces a nontrivial automorphism, $\nu$, on $LG_q = C_r \cup C_{i+r}$ for each $i\in[s]$,  which is non-distinguishing. But this contradicts  theorem \ref{lg3},  and so we are through. 
 \end{proof}
\section{Concluding Remarks}
\begin{itemize}
\item It is possible to consider other Levi graphs arising out of other projective geometries (affine planes, incidence bipartite graphs of $1$-dimensional subspaces versus $k$ dimensional subspaces in an $n$ dimensional vector space for some $k$ etc). Many of our results and methods work in those contexts as well and it should be possible to prove similar results there as well,  as long as the full automorphism group is not substantially larger. For instance, in the case of the incidence graphs of $k$ sets versus $l$-sets of $[n]$, it is widely believed (see  \cite{godsil}, chapter 1) that in most cases, the full automorphism group of the generalized Johnson graphs is indeed $S_n$ though it is not known with certainty.
\item As stated earlier, we believe that $\chi_D(LG_4)=3$ though we haven't been able to show the same. Similarly, we believe $\chi_D(LG_3)=4$. One can, by tedious arguments considering several cases,  show that a monochromatic $3$-coloring of $LG_3$ is not a proper distinguishing coloring. For details on what a monochromatic coloring is, see the Appendix for related details.
\item We were able to show $\chi_D(K_r^{\times n}) = r+1$ since in this case, all proper $r$ colorings of $K_r^{\times n}$ are of a specific type.  For an arbitrary (prime) graph $H$, it is not immediately clear if $\chi_D(H^{\times n})>\chi(H)$. It would be interesting to  find some  characterization of graphs $H$ with $\chi_D(H^{\times n}) =\chi(H)+1$ for large $n$.  
\item For a given $k \in \bN$, we obtained nontrivial examples of family of graphs $G$ with arbitrarily large chromatic number which have $\chi_D(G)>\chi(G)$ and with $|Aut(G)|$ reasonably small. It is not immediately clear if we can find infinite families of graphs $G$ with $\chi_D(G) > \chi(G)$ while $|Aut(G)| = O(|G|)$. If we were to hazard a guess, our immediate guess would be no but we do not have sufficient reason to justify the same.
\item While we have attempted to construct non-trivial families of bipartite graphs with large distinguishing chromatic number, it would be interesting to construct nontrivial examples of graphs with arbitrary chromatic number, and arbitrarily large distinguishing chromatic number. 
\end{itemize}

\newpage
\section{Appendix}
\subsection{The Levi Graph $LG_2$}
Firstly, we remark that the upper bound $\chi_D(G)\leq 2\Delta-2$ whenever $G$ is bipartite and $G \ncong K_{\Delta-1,\Delta}, K_{\Delta, \Delta },$ which appears in
\cite{LS}, gives $\chi_D(LG_q)\leq 2q$. In particular, $\chi_D(LG_2)\leq 4$. We shall show that in fact $\chi_D(LG_2) = 4$. 

We first set up some notation,  let $\{e_1, e_2, e_3 \}$ be the standard basis of the vector space $V$
with $e_1 = (1, 0, 0), e_2 = (0, 1, 0)$ and $e_3 = (0, 0, 1).$  For $g, h, k \in F_q,$ a vector $v \in V$ is denoted by $(g, h, k)$ if $ v = ge_1 + he_2 + ke_3.$
 A point $p\in\Pa $ is denoted by  $(g, h, k)$ if $p = < ge_1 + he_2 + ke_3 >.$ Thus, there are $q^2$ points in the form $(1, h, k)$ such that $ h,k \in \bF_q,$  $q$ points in the
form of $(0, 1, k)$ such that $k \in \bF_q$ and finally the point $(0, 0, 1)$ to account for a total of $q^2 + q + 1$ points in \PG.

We start with the following definition.
\begin{defn} 
A  coloring of the Levi graph is said to be {\bf Monochromatic} if all the vertices in one set of the vertex partition have the same color.
\end{defn}
\begin{lem}\label{mono}
$LG_2$ does not have a proper distinguishing monochromatic $3$-coloring.
\end{lem}
\begin{proof}
Assume that $LG_2$ has a proper distinguishing monochromatic $3$-coloring. Without loss of generality let the line set $\mathcal{L}$ be colored with a single color, say red.  Call the remaining two colors blue and green, say, which are the colors assigned to the vertices in $\Pa$. We shall refer to the set of points that are assigned a particular color, say green, as the  color class $Green$. By rank of a color class $\C$ (denoted $r(\C)$), we mean the rank of the vector subspace generated by $\C$. Observe that a nontrivial  linear map $T$ that fixes the color class $Green$, must necessarily also fix the color class $Blue$, so any such linear map would correspond to an automorphism that preserves each color class.\\
For any $2$-coloring of $\Pa$ (which has $7$ points), one of the two color class has fewer than four points. Without loss of generality, assume that this is the color class $Green$. Firstly, if $r(Green)\le 2$ then consider a basis $B$ of $V$ which contains a maximal linearly independent set of points in color class $Green$. If $r(Green)=2$, then the linear map $T$ obtained by swapping the elements of the color class $Green$ in $B$, and fixing every other basis element is a non-trivial linear transformation of $V$ which necessarily fixes the color class $Green$. If $r(Green) = 1$, then consider the map $T$ which fixes the green point of $B$ and swaps the other two (necessarily Blue) is a nontrivial linear transform that fixes the color class $Green$. Finally, if $r(Green) = 3$, then  let $T$ be the map that swaps two of them and fixes the third. Again, this map is a nontrivial linear map that fixes every color class. 
\end{proof}
We now set up some notation. Denote the Points in $LG_2$ by $\{ e_1, e_2, e_3, e_1+e_2, e_1
+ e_3, e_2 + e_3, e_1 + e_2 + e_3 \}$ and denote the lines in the following way:
\begin{enumerate}
 \item $l_1$ : $\langle e_1, e_2\rangle$ the line ( two dimensional subspace) spanned by $e_1$ and  $e_2$.
 \item $l_2$ : $\langle e_1, e_3\rangle$.
 \item $l_3$ : $\langle e_2, e_3\rangle$.
 \item $l_4$ : $\langle e_1, e_2 + e_3\rangle$.
 \item $l_5$ : $\langle e_2, e_1 + e_3\rangle$.
 \item $l_6$ : $\langle e_3, e_1 + e_2\rangle$.
 \item $l_7$ : $\langle e_1 + e_3, e_2 + e_3\rangle$.
\end{enumerate}
\begin{thm}\label{dl2n3} 
$\chi_D(LG_2)=4$. 
\end{thm}
\begin{proof} By the remark at the beginning of the section, we have $\chi_D(LG_2)\le 4$, so it suffices to show $\chi_D(LG_2)>3$.
We first  claim that  if $LG_2$ has a proper distinguishing $3$-coloring, then three linearly independent points (points corresponding to three linearly independent vectors) get the same color.\\
Suppose the claim is false. Then each monochrome set $\C$ of points satisfies $r(\C)\le 2$. Since any set of four points contains three linearly independent points and $|V(LG_2)|=7$, a $3$-coloring yields a monochrome set of points of size exactly three. Denote this set by  $E$ and observe that $E$ in fact corresponds to a line $l_E\in\List$. Since any two lines intersect, no line is colored the same as the points of $E$. If $p,p'\in\Pa\setminus E$ are colored differently, then the line $l_{p,p'}$ cannot be colored by any of the three colors contradicting the assumption.  Consequently,  every point in $\Pa\setminus E$ must be colored the same if the coloring were to be proper. But then this gives a color class with four points which contains three linearly independent points contradicting that the claim was false.
Without loss of generality, suppose $e_1, e_2, e_3 $ are all colored red. Since $l_7$ contains the points  $ e_1 + e_2, e_2 + e_3$ and $e_1 + e_3,$ these three points cannot all have different colors. Hence at least two of these three points are in the same color class. 

Without loss of generality, assume that $e_1 + e_2$ and $e_2 + e_3$ have the same color. 
   Now observe that the map $\sigma$ defined by $\sigma(e_1)=e_3,\sigma(e_3)=e_1, \sigma(e_2)=e_2$, induces an automorphism of $LG_2$ that fixes every color class within $\Pa$. Furthermore $\sigma$ swaps $l_1$ with $l_3$ and $l_4$ with $l_6$ and fixes all the other lines. If  the sets of lines $\{l_4,l_6\}$ and $\{l_1,l_3\}$ are both monochrome  in $\List$,  then note that $\sigma$ fixes every color class contradicting that the coloring in question is distinguishing. Thus we consider the alternative, i.e.,  the possibilities that the lines $l_1$ and $l_3$ (resp. $l_4$ and $l_6$) are in different color classes, and in each of those cases produce a non-trivial automorphism fixing every color class.
   
\textit{Case I : } $l_4$ and $l_6$ have different colors, say blue and green respectively.
 In this case, the point set witnesses at most two colors and none of the points of $\Pa\setminus\{e_1 + e_3\}$ can be colored blue or green. Moreover, by
lemma \ref{mono}, all the seven points cannot be colored red (note that $e_1,e_2,e_3$ are colored red). Consequently,  $e_1 + e_3$ is 
colored, say blue, and all the other points are colored red. The  $l_7,$ $l_5$ and $l_2$ are all colored green since all these three
lines contain the point $e_1 + e_3.$
As mentioned above, we shall in every case that may arise, describe a non-trivial automorphism $\sigma$ that fixes each color class. As before, we shall only describe its action on the set $\{e_1,e_2,e_3\}$.\\
\textit{Sub case 1 : } $l_1$ is colored blue. Then 
 $\sigma(e_1) = e_1, \sigma(e_2)=e_2 + e_3,\sigma(e_3)= e_3$ fixes $e_1 + e_3,$ swaps $l_1$ with $l_4$ and fixes $l_3.$ Consequently, it fixes every color class.\\
\textit{Sub case 2 : }$l_1$ is colored green and $l_3$ is colored blue.
In this case,  $\sigma(e_1)= e_2, \sigma(e_2)= e_1,\sigma(e_3)= e_1 + e_2 + e_3$ does the job.
\textit{ Sub case 3 : }  $l_1$ and $l_3$ are both colored green. In this case, the only line which is colored blue is $l_4.$ Then  $\sigma(e_1)= e_2 + e_3, \sigma(e_2)= e_2,\sigma(e_3)= e_1 + e_2,$ does the job.\\
From the above it follows that $l_4$ and $l_6$ cannot be in different color
classes. So, we now consider the other possibility, namely that $l_1$ and $l_3$ are in different color classes.
 
\textit{ Case II: } $l_6$ and $l_4$ have the same color but $l_1$ and $l_3$ are in different color classes, say blue and green respectively.
Here we first note that $e_1 + e_2$ and $e_2 + e_3$ are necessarily red because they belong to $l_1$ and $l_3$ respectively. Again, we are led to three subcases:\\
\textit{ Sub case 1 : } $e_1 + e_3$ and $e_1 + e_2 + e_3$ are both colored blue. Here, it is a straightforward check to see that every $l\neq l_1$ is colored green. Then, one can check that $\sigma(e_1)= e_1 + e_2, \sigma(e_2)=e_2, \sigma(e_3)=e_3$  fixes every color class.\\
\textit{ Sub case 2 : } The point $e_1 + e_3$ is colored red and $e_1+ e_2 + e_3$ is colored blue. Again, one can check in a straightforward manner, that for all $3\le i\le 6$, $l_i$ is colored green. If $l_2$ is blue then  $\sigma(e_2)=e_3,\sigma(e_3)=e_2, \sigma(e_1)=e_1$ does the job. If $l_2$ is colored green, $\sigma(e_1)=e_2,\sigma(e_2)=e_1, \sigma(e_3)=e_3$ does the job.\\
\textit{ Sub case 3 : } $e_1 + e_2 + e_3$ is colored red and $e_1 + e_3$ is colored blue. Here we first observe that $l_2,l_3,l_5,l_7$ are all necessarily green. Also, by the underlying assumption (characterizing Case II), $l_4,l_6$ bear the same color. In this case,  $\sigma(e_1)=e_1 + e_2, \sigma(e_3)= e_2 + e_3, \sigma(e_2)=e_2$, does the job.
This exhausts all the possibilities, and hence we are through.
\end{proof}
\subsection{The Levi graph $LG_3$}
As remarked earlier, it is not too hard to show that $\chi(LG_q)\le 6$, so the same holds for $q=3$ as well. The next proposition shows an improvement on this result.
\begin{thm}
 $\chi_D(LG_3) \leq 5.$ 
\end{thm}
\begin{proof}
As indicated earlier  we denote the points $p\in\Pa$ as mentioned in the beginning of this section. A line corresponding to the subspace $\{ (x, y, z) \in \mathcal{P}: ax + by + cz = 0\}$ is denoted ${\bf(a,b,c)}$.   We color the graph using the colors $1,2,3,4,5$ as in figure \ref{lg3} (the color is indicated in a rectangular box corresponding to the vertex)
\begin{figure}[h!]
 \centering
  \includegraphics[width=4in,height=3in]{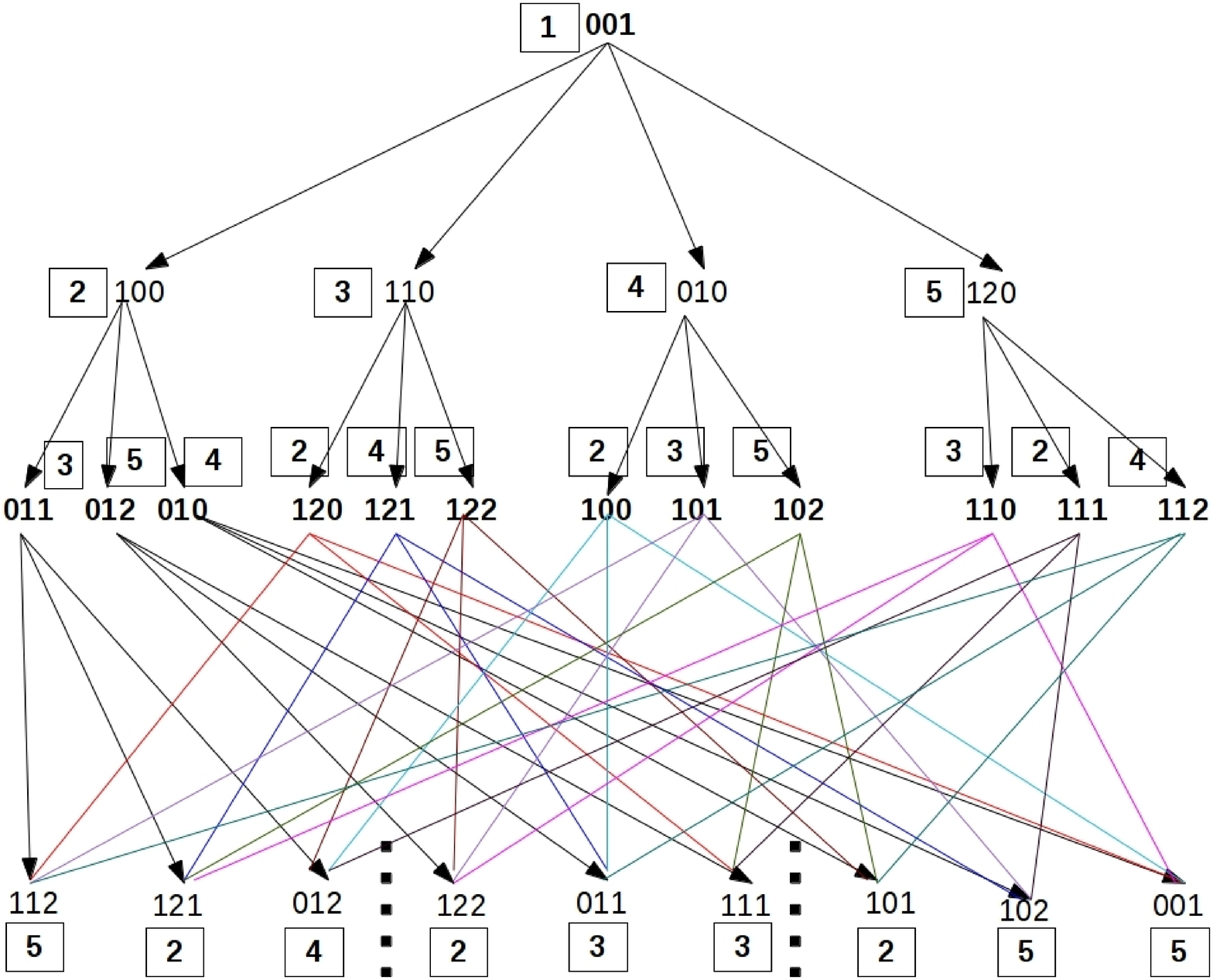}
 \caption{$LG_3$}\label{lg3}
\end{figure} 
It is straightforward to check that the coloring is proper.  For an easy check we provide below, a table containing adjacencies of  each $p\in\Pa$. \\ \\
\begin{table}[h]
\begin{tabular}{|l | c | c | c | c |c | c | c | c | c | c | c | c | c | r |}
  \hline                       
  Points $\rightarrow$ & 100 & 110 & 010 & 120 & 112 & 121 & 012 & 122 & 011 & 111 & 101 & 102 & 001\\
  \hline
 Lines &{\bf001} &{\bf 001} & {\bf 001} & {\bf 001} &{\bf 011} &{\bf 011} &{\bf 011} & {\bf 012} &{\bf 012} &{\bf 012} &{\bf 010} & {\bf 010} &{\bf 010}\\
  $\downarrow$ & {\bf 011} & {\bf 120} & {\bf 100} &{\bf 110} & {\bf 120} & {\bf 121} & {\bf 122} & {\bf 122} & {\bf 121} & {\bf 120} &{\bf 122} & {\bf 121} &{\bf 120}\\
  &{\bf 012} &{\bf 121} & {\bf 101} &{\bf 111} &{\bf 101} & {\bf 102} &{\bf 100} & {\bf 101} &{\bf 100} & {\bf 102} &{\bf 102} & {\bf 101} & {\bf 100}\\
  & {\bf 010} & {\bf 122} & {\bf 102} & {\bf 112} & {\bf 112} & {\bf 110} & {\bf 111} & {\bf 110} & {\bf 112} & {\bf 111} & {\bf 112} & {\bf 111} & {\bf 110}\\
  \hline  
\end{tabular}
\end{table}
Here the first row lists all the points in the projective plane of order $3$. The column corresponding to the vertex $p\in\Pa$ lists the set of lines $l\in\List$ such that $p\in l$, so that the columns are the adjacency lists for the vertices in $\Pa$. 
To see that this coloring is distinguishing, firstly, observe that the line ${\bf 001}$ is the only vertex with color $1.$ Therefore, any automorphism $\phi$ that fixes every color class necessarily fixes this line. Consequently,  the points on ${\bf 001}$ are mapped by $\phi$ onto themselves. Since each point on ${\bf 001}$ bears a different color, it follows that $\phi$ fixes each $p\in{\bf 001}$. In particular, for $1 \leq i \leq 4,$ $\phi$ maps each set $\{ l_{i1}, l_{i2}, l_{i3} \}$ onto itself. Here,  $\{l_{ij}, 1 \leq j \leq 3\}$ denotes the set of lines adjacent to the $i^{th}$ point of ${\bf 001}.$ But again note that by the coloring indicated, the vertices $l_{ij}$ and $l_{ij'}$ have different colors for each $i$, so $\phi(l_{ij})=l_{ij}$ for each pair $(i,j)$ with $1\le i\le 4, 1\le j\le 3$. Now it is a straightforward check to see that $\phi=I$. 
\end{proof}

 \end{document}